\documentclass[11pt, a4paper, oneside]{article} 

\usepackage[utf8x]{inputenc}
\usepackage[T1]{fontenc}

\usepackage{amsmath}
\usepackage{amsthm}	  	  
\usepackage{amsfonts}
\usepackage{amssymb}
\usepackage{hyperref}
\usepackage{mathtools}
\usepackage{graphicx}
\usepackage[colorinlistoftodos]{todonotes}
\usepackage{tikz}
\usetikzlibrary{matrix,arrows,decorations.pathmorphing}

\usepackage[square, numbers, comma, sort&compress]{natbib} 

\theoremstyle{plain}
\newtheorem{theorem}{Theorem}[section]
\newtheorem{lemma}[theorem]{Lemma}
\newtheorem{claim}[theorem]{Claim}
\newtheorem{cor}[theorem]{Corollary}

\newtheorem{fact}[theorem]{Fact}

\theoremstyle{definition}
\newtheorem{definition}[theorem]{Definition}
      
\theoremstyle{remark}

\newcommand{\lemref}[1]{\hyperref[#1]{Lemma \ref*{#1}}}
\newcommand{\clmref}[1]{\hyperref[#1]{Claim \ref*{#1}}}
\newcommand{\thmref}[1]{\hyperref[#1]{Theorem \ref*{#1}}}
\newcommand{\propref}[1]{\hyperref[#1]{Proposition \ref*{#1}}}
\newcommand{\corref}[1]{\hyperref[#1]{Corollary \ref*{#1}}}
\newcommand{\defref}[1]{\hyperref[#1]{Definition \ref*{#1}}}
\newcommand{\remref}[1]{\hyperref[#1]{Remark \ref*{#1}}}
\newcommand{\conjref}[1]{\hyperref[#1]{Conjecture \ref*{#1}}}
\newcommand{\factref}[1]{\hyperref[#1]{Fact \ref*{#1}}}

\newcommand{\Ker}{\mathrm{Ker}}
\newcommand{\Imm}{\mathrm{Im}}

\makeatletter
\newcommand*{\defeq}{\mathrel{\rlap{%
                     \raisebox{0.27ex}{$\m@th\cdot$}}%
                     \raisebox{-0.27ex}{$\m@th\cdot$}}%
                     =}

\numberwithin{equation}{section}

\begin{document}

\title{Semi-free subgroups of a profinite surface group}

\author{Matan Ginzburg and Mark Shusterman}
\date{}

\maketitle

\abstract{

We show that every closed normal subgroup of infinite index in a profinite surface group $\Gamma$ is contained in a semi-free profinite normal subgroup of $\Gamma$. This answers a question of Bary-Soroker, Stevenson, and Zalesskii.

}

\section{Introduction}

The classical theorem of Nielsen and Schreier states that every subgroup of a free group is free. Trying to extend this result to profinite groups fails, for example $\mathbb{Z}_2 \leq \hat{\mathbb{Z}}$. This naturally gives rise to the question of finding conditions upon which a subgroup of a free profinite group is free.

Some results in this direction are known, for instance, Melnikov's characterization of normal subgroups of free profinite groups, and Haran's diamond theorem. Results of slightly different flavor have been obtained by Shusterman, in \cite{Mark} where, for example, the following is proven.
\begin{theorem} \label{Mark's theorem}

Let F be a nonabelian finitely generated free profinite group, and let $H \leq_c F$ be a closed subgroup of infinite index in F. Then there exists a free profinite subgroup $H \leq L \leq_c F$ of rank $\aleph_0$.

\end{theorem}
In particular, weakly maximal subgroups are free.

In this work we consider an analog of the above for profinite surface groups. These groups show up as \'{e}tale fundamental groups of curves over an algebraically closed field of characteristic 0. 

We will be interested in semi-free profinite subgroups (of profinite surface groups), a notion introduced in \cite{DanandLior}, where it is shown that a group is free profinite if and only if it is projective and semi-free. 
As shown in \cite{Zal}, projectivity of a subgroup of a profinite surface group is equivalent to a simple condition on its index (as a supernatural number).
Henceforth, we will focus on semi-freeness.

Our main result is the following.

\begin{theorem}

Let $N \lhd_c \Gamma_g$ be a normal subgroup of infinite index in a profinite surface group of genus $g \geq 2$. Then there exists a semi-free profinite subgroup $M \lhd_c \Gamma_g$ that contains $N$.

\end{theorem}

This answers a question raised by Bary-Soroker, Stevenson, and Zalesskii (in \cite[Remark 4.1]{LiorandPavel}) who used their diamond theorem to establish the special case where $\Gamma_g/N$ is not hereditarily just infinite.

Our method also gives the following analog of the aforementioned results of Shusterman. 

\begin{theorem}

Weakly maximal subgroups of profinite surface groups are semi-free profinite.

\end{theorem}

Weakly maximal subgroups were also studied in the context of branch groups, for instance in \cite{Grigorchuk}.

\section{Preliminaries}

Here we give the basic definitions and claims that will be used in the rest of this paper. We will work in the category of profinite groups, namely we assume that every subgroup is closed, every homomorphism is continuous, and so on.

\begin{definition}
For a finitely generated profinite group $G$ we denote by $d(G)$ the minimal size of a generating set of $G$.
\end{definition}

\begin{definition}
An infinite profinite group $G$ is called \textbf{just infinite} if for every $\{1\} \neq N \lhd_c G$ the quotient $G/N$ is finite. Equivalently, every non trivial normal subgroup of $G$ is open.
\end{definition}

\begin{definition}
An infinite profinite group is called \textbf{hereditarily just infinite} if every open normal subgroup of it is just infinite.
\end{definition}

\begin{definition}
Let $H$ be a closed subgroup of infinite index in a profinite group $G$. We say that $H$ is \textbf{weakly maximal} in $G$ if every $H \lneq K \leq_c G$ is open.
\end{definition}


\begin{definition}
Given groups $G, A, B$ and surjective homomorphisms $$\alpha: A \to B,  \beta: G \to B$$ we define the \textbf{embedding problem} $\mathcal{E}(G,A,B,\alpha,\beta)$ as the problem of finding a homomorphism $\varphi: G \to A$ such that $\beta = \alpha \circ \varphi$. Such a homomorphism $\varphi$ is called a \textbf{solution} to the problem. If moreover $\varphi$ is surjective then it is called a \textbf{proper solution}.
\end{definition}

\begin{definition}
An embedding problem $\mathcal{E}(G,A,B,\alpha,\beta)$ is called \textbf{finite} if $A$ is finite. Note that since $\alpha$ is surjective, $B$ is also finite.
\end{definition}

\begin{definition}
An embedding problem $\mathcal{E}(G,A,B,\alpha,\beta)$ is called \textbf{split} if there is a homomorphism $\gamma: B \to A$ such that $\alpha \circ \gamma = \mathrm{id}_B$.
In such a case we have $A \cong \Ker(\alpha) \rtimes B.$
\end{definition}


\begin{definition}
A profinite group $G$ of rank $\aleph_0$ is called \textbf{semi-free} if every finite split embedding problem $\mathcal{E}(G,A,B,\alpha,\beta)$ has a proper solution.
\end{definition}

\begin{definition}
The \textbf{profinite surface group of genus $g$} is the group given by the profinite presentation $$\Gamma_g = \langle x_1,\dots,x_g,y_1,\dots,y_g \ | \  \prod_{i=1}^g [x_i,y_i] = 1 \rangle.$$
\end{definition}

\begin{definition}
Let $\Gamma_g$ be a profinite surface group of genus $g$. A \textbf{surface basis} of $\Gamma_g$ is a set of generators $x_1,\dots,x_g,y_1,\dots,y_g$
such that $$\Gamma_g \cong \langle x_1,\dots,x_g,y_1,\dots,y_g \ | \  \prod_{i=1}^g [x_i,y_i] = 1 \rangle.$$
\end{definition}

\begin{fact} \label{surface subgroup}
An open subgroup $H$ of a genus $g$ profinite surface group $\Gamma$ is a profinite surface group of genus $[\Gamma : H](g-1)+1$.
\end{fact}

\begin{claim} \label{onto kernel claim}
Let $\mathcal{E}(G,A,B,\alpha,\beta)$ be an embedding problem and let $\varphi: G \to A$ be a solution of $\mathcal{E}$. If $\Ker(\alpha) \subseteq \Imm(\varphi)$ then $\varphi$ is a proper solution.
\end{claim}

\begin{proof}
Let $a \in A$. Denote $b = \alpha(a)$ and let $g \in G$ be such that $\beta(g) = b$. Then $$\alpha(a) = \beta(g) = \alpha(\varphi(g)),$$ therefore $$a\varphi(g)^{-1} \in \Ker(\alpha) \subseteq \Imm(\varphi).$$ 
As $\Imm(\varphi)$ is a subgroup, we conclude that $a \in \Imm(\varphi)$ so $\Imm(\varphi) = A$.
\end{proof}

\section{Semi-free subgroup}

We need the following variant of \cite[Lemma 6.1]{Pavel}.

\begin{lemma} \label{Pavel lemma}
Let $$\Gamma = \langle x_i, y_i \mid \prod_{i=1}^g [x_i,y_i] = 1 \rangle$$ be a profinite surface group, and let $N \leq_c \Gamma$. Consider the diagram
$$\begin{tikzpicture}[scale=1.5]
\node (A) at (1,2) {$\Gamma$};
\node (B) at (1,1) {$N$};
\node (C) at (0,0) {$A$};
\node (D) at (1,0) {$B$};
\path[->,font=\scriptsize,>=angle 90]
(B) edge node[right]{$\beta$} (D)
(C) edge node[above]{$\alpha$} (D)
(A) edge[bend left] node[right]{$\bar{\beta}$} (D);
\end{tikzpicture}$$
where $A,B$ are finite groups, $\alpha$ is a surjection, and  $\bar{\beta}_{|N}= \beta$ is a surjection as well. Denote by $$\mathcal{E} = \mathcal{E}(N,A,B,\alpha,\beta), \quad \bar{\mathcal{E}} = \mathcal{E}(\Gamma,A,B,\alpha,\bar{\beta})$$ the two finite embedding problems in the above diagram.

Let $\varphi \colon \Gamma \to A$ be a solution to $\bar{\mathcal{E}}$, and set $$K := \Ker(\alpha), \quad n := |K\varphi(\Gamma)|, \quad s := d(K).$$ Suppose that $g \geq sn+s$ and that $$\forall \  1 \leq i,j \leq sn+s \ \varphi(x_i) = \varphi(x_j), \  \varphi(y_i) = \varphi(y_j),\  x_ix_1^{-1} \in N.$$ Then ${\mathcal{E}}$ admits a proper solution.
\end{lemma}

\begin{proof}
Choose a set of generators $B := \{k_1, \dots, k_s\}$ of $K$. Define $\eta(x_1) = \varphi(x_1)$ and
$$\eta(x_{in+j+1}) = \varphi(x_{sn+i+1})k_{i+1} = \varphi(x_1)k_{i+1}, \quad 0 \leq i \leq s-1, \ 1 \leq j \leq n.$$
Let $\eta$ coincide with $\varphi$ for all other generators of $\Gamma$, that is
$$\eta(x_i) = \varphi(x_i),\  \eta(y_j) = \varphi(y_j),\  sn + 1 < i \leq g,\  1 \leq j \leq g.$$

Since 
$$\forall\  0 \leq i \leq s-1 \ [\varphi(x_{sn+i+1})k_{i+1}, \varphi(y_{sn+i+1})] \in K\varphi(\Gamma)$$
we get that
$$\forall\  0 \leq i \leq s-1 \ [\eta(x_{in+2}), \eta(y_{in+2})]^n = 1 = [\varphi(x_{in+2}), \varphi(y_{in+2})]^n.$$
Therefore
\begin{equation*}
\begin{split}
&\prod_{i=1}^g [\eta(x_i), \eta(y_i)] = \\
&[\eta(x_1),\eta(y_1)] \cdot \prod_{i=0}^{s-1} [\eta(x_{in+2}), \eta(y_{in+2})]^n \cdot \prod_{i=sn+2}^g [\eta(x_i), \eta(y_i)] =\\
&[\varphi(x_1),\varphi(y_1)] \cdot \prod_{i=0}^{s-1} [\varphi(x_{in+2}), \varphi(y_{in+2})]^n \cdot \prod_{i=sn+2}^g [\varphi(x_i), \varphi(y_i)] =\\
&\prod_{i=1}^g [\varphi(x_i), \varphi(y_i)] = 1.
\end{split}
\end{equation*}
Thus $\eta$ extends to a homomorphism. As $x_1^{-1}x_i \in N$ for $1 \leq i \leq sn+s$, we conclude that $k_{j+1} = \eta(x_1^{-1}x_{jn+2}) \in \eta(N)$ for $0 \leq j \leq s-1$, hence $\eta(N)$ contains $K$, so  the result follows by invoking \clmref{onto kernel claim}.
\end{proof}

We also need the following generalization of \cite[Lemma 2.2]{LiorandPavel}.

\begin{lemma} \label{Lior and Pavel lemma}
Let $\Gamma_g$ be a profinite surface group of genus $g$, let $$\mathcal{E} = \mathcal{E}(\Gamma_g,A,B,\alpha,\beta)$$ be a finite split embedding problem, and suppose that $g \geq 2|A|^2|B|$. Then $\mathcal{E}$ has a proper solution, and $\Gamma_g$ has a surface basis $$x_1,\dots,x_g,y_1,\dots,y_g$$ such that
$$\forall \ 1 \leq i \leq m \ \ x_ix_1^{-1}, y_iy_1^{-1} \in \Ker(\beta),\quad m = \frac{2|A|^2}{|B|}.$$

Furthermore, if $N \leq_c \Gamma_g$ is such that $\beta(N) = B$ and $$\forall \  1 \leq i \leq m \ \ x_ix_1^{-1} \in N,$$ then the embedding problem $\bar{\mathcal{E}} = \mathcal{E}(N,A,B,\alpha,\beta_{|N})$ is properly solvable.
\end{lemma}

\begin{proof}
Let $z_1,\dots,z_g,w_1,\dots,w_g$ be a surface basis of $\Gamma_g$. Let $\gamma$ be a section of $\alpha$ and set $\varphi = \gamma \circ \beta$. As before, put $K = \Ker(\alpha)$, $s = d(K)$, and $n = |K\varphi(\Gamma_g)| = |A|$. Note that $$s \leq |K| = \frac{|A|}{|B|}$$ and thus $$sn+s \leq \frac{|A|^2+|A|}{|B|} \leq \frac{2|A|^2}{|B|} = m.$$

Each of the pairs $(\varphi(z_i),\varphi(w_i))$ can attain at most $|B|^2$ values, whence by the pigeonhole principle (since $g \geq m|B|^2$) there are $$1 \leq j_1 < j_2 < \dots < j_m \leq g$$ such that $$\varphi(z_{j_1}) = \dots = \varphi(z_{j_m}),\quad \varphi(w_{j_1}) = \dots = \varphi(w_{j_m}).$$

Suppose $j_1 \neq 1$, then
\begin{equation*}
\begin{split}
1 = \prod_{i=1}^g [z_i,w_i] &=
[z_{j_1},w_{j_1}] \Bigg[ \prod_{i=1}^{j_1-1} [z_i,w_i] \Bigg]^{[z_{j_1}, w_{j_1}]}\prod_{i=j_1+1}^g [z_i,w_i]\\
&= [z_{j_1},w_{j_1}] \cdot \prod_{i=1}^{j_1-1} [z_i^{[z_{j_1}, w_{j_1}]},w_i^{[z_{j_1}, w_{j_1}]}] \cdot \prod_{i=j_1+1}^g [z_i,w_i]
\end{split}
\end{equation*}
and so we can replace $\{z_i,w_i\}_{i=1}^g$ with a new surface basis $$z_{j_1},z_1^{[z_{j_1}, w_{j_1}]},\dots,z_{j_1-1}^{[z_{j_1}, w_{j_1}]},z_{j_1+1},\dots,z_g,$$
$$w_{j_1},w_1^{[z_{j_1}, w_{j_1}]},\dots,w_{j_1-1}^{[z_{j_1}, w_{j_1}]},w_{j_1+1},\dots,w_g.$$
By repeating this process with $j_2,\dots,j_m$ we obtain a surface basis $x_1,\dots,x_g,y_1,\dots,y_g$ of $\Gamma_g$ such that $x_i = z_{j_i}, y_i = w_{j_i}$ for $1 \leq i \leq m$, and so $$\varphi(x_1) = \dots = \varphi(x_m),\quad \varphi(y_1) = \dots = \varphi(y_m).$$

Since $sn + s \leq m$ we can apply \lemref{Pavel lemma} (with $N = \Gamma_g$ if necessary) and the result follows.
\end{proof}

\begin{cor} \label{free product lemma}
The finite split embedding problem
$$\begin{tikzpicture}[scale=1.5]
\node (B) at (0,1) {$\Gamma_g \amalg G$};
\node (C) at (-1,0) {$K \rtimes H$};
\node (D) at (0,0) {$H$};
\node (E) at (-2,0) {$K$};
\node (F) at (-3,0) {$1$};
\node (G) at (1,0) {$1$};
\path[->,font=\scriptsize,>=angle 90]
(B) edge node[right]{$\beta$} (D)
(C) edge node[above]{} (D)
(D) edge node[above]{$$} (G)
(E) edge node[above]{$$} (C)
(F) edge node[above]{$$} (E);
\end{tikzpicture}$$
has a proper solution once $g \geq 2|K|^2|H|^3$, $\Gamma_g$ is a profinite surface group of genus $g$, and $G$ is any profinite group.
\end{cor}

\begin{proof}
Denote $H_0 = \beta(\Gamma_g)$. According to \lemref{Lior and Pavel lemma} the finite split embedding problem
$$\begin{tikzpicture}[scale=1.5]
\node (B) at (0,1) {$\Gamma_g$};
\node (C) at (-1,0) {$K \rtimes H_0$};
\node (D) at (0,0) {$H_0$};
\node (E) at (-2,0) {$K$};
\node (F) at (-3,0) {$1$};
\node (G) at (1,0) {$1$};
\path[->,font=\scriptsize,>=angle 90]
(B) edge node[right]{$\beta_{|\Gamma_g}$} (D)
(C) edge node[above]{} (D)
(D) edge node[above]{$$} (G)
(E) edge node[above]{$$} (C)
(F) edge node[above]{$$} (E);
\end{tikzpicture}$$
has a proper solution $\varphi_1$. Let $\varphi_2: G \rightarrow K \rtimes H$ be the map defined by $$\varphi_2(f) = (1,\beta_{|G}(f))$$ for every $f \in G$. 
There exists a unique homomorphism $$\varphi: \Gamma_g \amalg G \rightarrow K \rtimes H$$ such that $$\forall \  \gamma \in \Gamma_g \  \forall \ f \in G \  \varphi(\gamma) = \varphi_1(\gamma), \quad \varphi(f) = \varphi_2(f).$$ 
By the universal property of free products, $\varphi$ is a solution to the original embedding problem.
Since $K$ is contained in $\Imm(\varphi_1)$ it is also contained in $\Imm(\varphi)$ so by \clmref{onto kernel claim} $\varphi$ is a proper solution.
\end{proof}

\begin{theorem}
Let $N \lhd_c \Gamma_g$ be a normal subgroup of the profinite surface group of genus $g \geq 2$ such that $\Gamma_g / N$ is hereditarily just infinite. Then $N$ is semi-free.
\end{theorem}

\begin{proof}
Let
$$\begin{tikzpicture}[scale=1.5]
\node (B) at (0,1) {$N$};
\node (C) at (-1,0) {$K \rtimes H$};
\node (D) at (0,0) {$H$};
\node (E) at (-2,0) {$K$};
\node (F) at (-3,0) {$1$};
\node (G) at (1,0) {$1$};
\path[->,font=\scriptsize,>=angle 90]
(B) edge node[right]{$\beta$} (D)
(C) edge node[above]{} (D)
(D) edge node[above]{$$} (G)
(E) edge node[above]{$$} (C)
(F) edge node[above]{$$} (E);
\end{tikzpicture}$$
be a finite split embedding problem for $N$. We shall prove it has a proper solution.

Using \cite[Lemma 1.2.5(c)]{MoshesBook}, 
we can extend our embedding problem to a subgroup $$N \leq F \lhd_o \Gamma_g$$ such that $$[\Gamma_g : F] \geq \frac{2|K|^2|H|^3+m-1}{g-1}$$ where $m = 2|K|^2|H|$. By \factref{surface subgroup}, $F$ is a profinite surface group of genus $$h = [\Gamma_g : F](g-1) + 1.$$ Note that $$h-m \geq 2|K|^2|H|^3.$$

Applying \lemref{Lior and Pavel lemma} to $F$, the extended embedding problem and $m$, we obtain a proper solution $\varphi$, and a surface basis $x_1,\dots,x_h,y_1,\dots,y_h$ of $F$ such that $$\forall \ 1 \leq j \leq m \ \ x_jx_1^{-1}, y_jy_1^{-1} \in \Ker(\beta).$$ If $x_ix_1^{-1} \in N$ for all $ 1 \leq i \leq m$ then by \lemref{Lior and Pavel lemma}, our original embedding problem has a proper solution.
Assume henceforth that $x_ix_1^{-1} \notin N$ for some $1 \leq i \leq m$.

The homomorphism $\beta$ factors modulo $$L = \langle x_2x_1^{-1},\dots,x_mx_1^{-1},y_2y_1^{-1},\dots,y_my_1^{-1} \rangle^F.$$ Note that $$F/L \cong \langle x_1, x_{m+1}, \dots,x_h,y_1, y_{m+1}, \dots,y_h \mid [x_1,y_1]^m \cdot \prod_{i=m+1}^h [x_i,y_i] = 1 \rangle$$ and that $$\beta \big( [x_1,y_1]^m \big) = \beta \big( [x_1,y_1] \big)^m = 1$$ so $\beta$ even factors through 
\begin{equation*}
\begin{split}
&F/\langle L, [x_1,y_1]^m \rangle^F \cong \\
&\langle x_i, y_i \mid \prod_{i=m+1}^h [x_i,y_i] = 1 \rangle \amalg \langle x_1,y_1 \mid [x_1,y_1]^m = 1 \rangle \cong \\
&\Gamma_{h-m} \amalg G
\end{split}
\end{equation*}
where $G = \langle x_1,y_1 \mid [x_1,y_1]^m = 1 \rangle$.

Let $$\psi \colon F \to F/\langle L, [x_1,y_1]^m \rangle^F$$ be the quotient map, and set $M = \psi(N)$. Since $\Gamma_g/N$ is hereditarily just infinite, $F/N$ is just infinite, so $\psi^{-1}(M)$ either equals $N$ or is open in $F$. The former is impossible as 
$$\exists \ 1 \leq i \leq m \ x_ix_1^{-1} \in \mathrm{Ker}(\psi) \subseteq \psi^{-1}(M), \quad x_ix_1^{-1} \notin N.$$
Hence $\psi^{-1}(M)$ is an open subgroup of $F$, so $M$ can be seen as an open subgroup of $\Gamma_{h-m} \amalg G$.

We can now write the original embedding problem as 
$$\begin{tikzpicture}[scale=1.5]
\node (B) at (0,2) {$N$};
\node (C) at (-1,0) {$K \rtimes H$};
\node (D) at (0,0) {$H$};
\node (E) at (-2,0) {$K$};
\node (F) at (-3,0) {$1$};
\node (G) at (1,0) {$1$};
\node (H) at (0,1) {$M$};
\path[->,font=\scriptsize,>=angle 90]
(B) edge node[right]{$\psi_{|N}$} (H)
(H) edge node[right]{$\bar{\beta}$} (D)
(C) edge node[above]{$$} (D)
(D) edge node[above]{$$} (G)
(E) edge node[above]{$$} (C)
(F) edge node[above]{$$} (E);
\end{tikzpicture}$$
so it is sufficient to properly solve it for $M$. 
Applying the Kurosh theorem (\cite[Theorem D.3.1]{PavelsBook}) we find that $M \cong \Gamma_t \amalg \overline{G}$ where $$t \geq h-m \geq 2|K|^2|H|^3$$ so by \corref{free product lemma} the desired proper solution exists.
\end{proof}


Repeating the above proof verbatim, one obtains the following.

\begin{theorem}
Let $\Gamma_g$ be a surface group and $N \leq_c \Gamma_g$ weakly maximal, then $N$ is semi-free profinite.
\end{theorem}

\section{Acknowledgments}
We would like to thank Lior Bary-Soroker and Pavel Zalesskii for useful remarks and discussions. The authors were partially supported by a grant of the Israel Science Foundation with cooperation of UGC no. 40/14.
Mark Shusterman is grateful to the Azrieli Foundation for the award of an Azrieli Fellowship.

\bibliographystyle{abbrv}
\bibliography{references}

\begin{thebibliography}{1}

\bibitem{DanandLior}
L.~Bary-Soroker, D.~Haran, and D.~Harbater.
\newblock Permanence criteria for semi-free profinite groups.
\newblock {\em Math. Ann.}, 348(3):539--563, 2010.

\bibitem{LiorandPavel}
L.~Bary-Soroker, K.~F. Stevenson, and P.~A. Zalesskii.
\newblock Subgroups of profinite surface groups.
\newblock {\em Math. Res. Lett.}, 18(3):459--471, 2011.

\bibitem{MoshesBook}
M.~D. Fried and M.~Jarden.
\newblock {\em Field arithmetic}, volume~11 of {\em Ergebnisse der Mathematik
  und ihrer Grenzgebiete. 3. Folge. A Series of Modern Surveys in Mathematics
  [Results in Mathematics and Related Areas. 3rd Series. A Series of Modern
  Surveys in Mathematics]}.
\newblock Springer-Verlag, Berlin, third edition, 2008.
\newblock Revised by Jarden.

\bibitem{Grigorchuk}
R.~I. Grigorchuk.
\newblock Branch groups.
\newblock {\em Mat. Zametki}, 67(6):852--858, 2000.

\bibitem{Pavel}
A.~l. Pacheco, K.~F. Stevenson, and P.~Zalesskii.
\newblock Normal subgroups of the algebraic fundamental group of affine curves
  in positive characteristic.
\newblock {\em Math. Ann.}, 343(2):463--486, 2009.

\bibitem{PavelsBook}
L.~Ribes and P.~Zalesskii.
\newblock {\em Profinite groups}, volume~40 of {\em Ergebnisse der Mathematik
  und ihrer Grenzgebiete. 3. Folge. A Series of Modern Surveys in Mathematics
  [Results in Mathematics and Related Areas. 3rd Series. A Series of Modern
  Surveys in Mathematics]}.
\newblock Springer-Verlag, Berlin, second edition, 2010.

\bibitem{Mark}
M.~Shusterman.
\newblock Free subgroups of finitely generated free profinite groups.
\newblock {\em J. Lond. Math. Soc. (2)}, 93(2):361--378, 2016.

\bibitem{Zal}
P.~A. Zalesskii.
\newblock Profinite surface groups and the congruence kernel of arithmetic
  lattices in {${\rm SL}_2({\bf R})$}.
\newblock {\em Israel J. Math.}, 146:111--123, 2005.

\end{thebibliography}

\bigskip 

Raymond and Beverly Sackler School of Mathematical Sciences, Tel-Aviv University, Tel-Aviv, Israel.

\bigskip

matan.ginzburg@gmail.com

\bigskip

markshus@mail.tau.ac.il

\end{document}